\documentclass[10pt]{article}
\usepackage{amssymb}
\usepackage{amsthm}
\usepackage{latexsym}
\usepackage{comment}
\usepackage{hyperref}
\usepackage{xypic}
\usepackage{graphicx}

\title{Dyadic Torsion of Elliptic Curves}
\author{Jeffrey Yelton}

\newtheorem{thm}{Theorem}[section]
\newtheorem{prop}[thm]{Proposition}

\newtheorem{lemma}[thm]{Lemma}

\newtheorem{dfn}[thm]{Definition}

\newtheorem{rmk}[thm]{Remark}

\newcommand{\qq}{\mathbb{Q}}
\newcommand{\zz}{\mathbb{Z}}

\newcommand{\Gal}{\mathrm{Gal}}
\newcommand{\GL}{\mathrm{GL}}
\newcommand{\SL}{\mathrm{SL}}
\newcommand{\Aut}{\mathrm{Aut}}

\begin{document}

\maketitle

\begin{abstract}

Let $k$ be a field of characteristic $0$, and let $\alpha_{1}$, $\alpha_{2}$, and $\alpha_{3}$ be algebraically independent and transcendental over $k$.  Let $K$ be the transcendental extension of $k$ obtained by adjoining the elementary symmetric functions of the $\alpha_{i}$'s.  Let $E$ be the elliptic curve defined over $K$ which is given by the equation $y^{2} = (x - \alpha_{1})(x - \alpha_{2})(x - \alpha_{3})$.  We define a tower of field extensions $K = K_{0}' \subset K_{1}' \subset K_{2}' \subset ...$ by giving recursive formulas for the generators of each $K_{n}'$ over $K_{n - 1}'$.  We show that $K_{\infty}'$ is a certain central subextension of the field $K(E[2^{\infty}]) := \bigcup_{n = 0}^{\infty} K(E[2^{n}])$, and a generator of $K(E[2^{\infty}])$ over $K_{\infty}'(\mu_{2})$ is given.  Moreover, if we assume that $k$ contains all $2$-power roots of unity, for each $n$, we show that $K(E[2^{n}])$ contains $K_{n}'$ and is contained in a certain quadratic extension of $K_{n + 1}'$.

\end{abstract}

\section{Introduction}

Let $k$ be a field of characteristic $0$ which contains all $2$-power roots of unity.  Let $K$ be the transcendental extension of $k$ obtained by adjoining the coefficients of the cubic polynomial $(x - \alpha_{1})(x - \alpha_{2})(x - \alpha_{3})$, where $\alpha_{1}$, $\alpha_{2}$, and $\alpha_{3}$ are independent and transcendental over $k$.  Fix an algebraic closure $\bar{K}$ of $K$.  Suppose that $E$ is the elliptic curve over $K$ given by the Weierstrass equation
\begin{equation}\label{elliptic model} y^{2} = (x - \alpha_{1})(x - \alpha_{2})(x - \alpha_{3}). \end{equation}
For any integer $n \geq 0$, let $E[2^{n}]$ be the subgroup of $E(\bar{K})$ of $2^{n}$-torsion points, and let $K_{n}$ be the extension of $K$ over which they are defined.  (Note that $K_{0} = K$ and $K_{1} = K(\alpha_{1}, \alpha_{2}, \alpha_{3})$.)  Further, denote by $E[2^{\infty}]$ the subgroup of all $2$-power torsion points and denote by $K_{\infty}$ the minimal (algebraic) extension of $K$ over which they are defined.

In this paper, 
$$T_{2}(E) := \lim_{\leftarrow n} E[2^{n}]$$
will denote the $2$-adic Tate module of $E$; it is a free $\zz_{2}$-module of rank $2$.  Let 
$$V_{2}(E) := T_{2}(E) \otimes \qq_{2}.$$
Then $V_{2}(E)$ is a $2$-dimensional vector space over $\qq_{2}$ which contains the rank-$2$ $\zz_{2}$-lattice $T_{2}(E)$.  Let $\mathcal{L}$ be the set of all $\zz_{2}$-lattices $\Lambda \subset V_{2}(E)$ such that $\Lambda \supset T_{2}(E)$ but $\Lambda \not \supset \frac{1}{2}T_{2}(E)$.  Set $\Lambda_{0} = T_{2}(E)$.  There is an obvious bijection between $\mathcal{L}$ and the set of cyclic subgroups of $E[2^{\infty}]$, given by sending an element $\Lambda \in \mathcal{L}$ to $\Lambda / \Lambda_{0}$, which is canonically identified with a cyclic subgroup $N < E[2^{n}]$ for any $n$ such that $\Lambda \subset \frac{1}{2^{n}}\Lambda_{0}$.  For each $n \geq 0$, we denote by $\mathcal{L}_{n}$ the subset of $\mathcal{L}$ consisting of lattices $\Lambda$ such that $\Lambda / \Lambda_{0}$ is cyclic of order $2^{n}$ (or equivalently, $\Lambda$ corresponds to a maximal cyclic subgroup of $E[2^{n}]$ under the above bijection), and let $\mathcal{L}_{\leq n} := \bigcup_{0 \leq m \leq n} \mathcal{L}_{m}$ and $\mathcal{L}_{\geq n} := \bigcup_{m \geq n} \mathcal{L}_{m}$.

We now endow $\mathcal{L}$ with the structure of a graph by saying that two lattices in $\mathcal{L}$ are connected by an edge if one is contained in the other and the induced quotient of $\zz_{2}$-modules is isomorphic to $\zz / 2\zz$.  It is easy to show that this graph is isomorphic to a certain Bruhat-Tits tree described in \cite{serre2003trees}, \S1.1.  In particular, $\mathcal{L}$ is a $3$-regular tree.  We designate $\Lambda_{0}$ as the root, and observe that for each $n \geq 0$, $\mathcal{L}_{n}$ is the set of all vertices of distance $n$ from the root.  The following properties immediately follow.

\begin{prop}\label{prop properties of L}

a) Every vertex $\Lambda \in \mathcal{L}_{n}$ with $n \geq 1$ has a unique ``parent" vertex; that is, a vertex $\tilde{\Lambda} \in \mathcal{L}_{n - 1}$ which is connected to $\Lambda$ by an edge.  Equivalently, $\tilde{\Lambda}$ is the unique lattice in $\mathcal{L}$ such that $\Lambda \supset \tilde{\Lambda}$ and $\Lambda / \tilde{\Lambda} \cong \zz / 2\zz$.

b) Every vertex $\Lambda \in \mathcal{L}_{n}$ with $n \geq 2$ has a unique ``twin" vertex; that is, a vertex $\Lambda' \in \mathcal{L}_{n}$ such that $\Lambda \neq \Lambda'$ but $\Lambda$ and $\Lambda'$ have the same parent.

\end{prop}

There are exactly $3$ vertices in $\mathcal{L}_{1}$, corresponding to the $3$ subgroups of order $2$ of $E[2]$, which are generated by the $2$-torsion points $(\alpha_{1}, 0)$, $(\alpha_{2}, 0)$, and $(\alpha_{3}, 0)$.  For $i = 1, 2, 3$, write $\Lambda(\alpha_{i})$ for the vertex in $\mathcal{L}_{1}$ corresponding to the subgroup $\langle (\alpha_{i}, 0) \rangle < E[2]$.  For any $\Lambda \in \mathcal{L}_{\geq 1}$, write $\tilde{\Lambda}$ for its parent vertex, and for any $\Lambda \in \mathcal{L}_{\geq 2}$, write $\Lambda'$ for its twin vertex, as in Proposition \ref{prop properties of L}.  For $\Lambda = \Lambda(\alpha_{i}) \in \mathcal{L}_{1}$, let $\Lambda' = \Lambda(\alpha_{i + 1})$, where the index $i$ is considered as an element of $\zz / 3\zz$.

In order to state the main theorem, we need to define an infinite algebraic extension of $K$ which is obtained by adjoining generators corresponding to each vertex of $\mathcal{L}$.  These generators have to be defined recursively, so that for any $n \geq 2$, the generator corresponding to a vertex in $\mathcal{L}_{n}$ is given by an algebraic function of the generators corresponding to certain vertices in $\mathcal{L}_{n - 1}$.  This motivates the following definition.

\begin{dfn}\label{dfn decoration}

A \textit{decoration} on the tree $\mathcal{L}$ is a function $\Phi : \mathcal{L}_{\geq 1} \to \bar{K}$ with the following three properties:

\indent I) For any $\Lambda \in \mathcal{L}_{\geq 1}$, $\Phi(\Lambda) \neq \Phi(\Lambda')$.

\indent II) For $i \in \zz / 3\zz$, $\Phi(\Lambda(\alpha_{i})) = \alpha_{i + 1} - \alpha_{i + 2}$.

\indent III) For every $\Lambda \in \mathcal{L}_{2}$, $\Psi(\Lambda)$ is a root of the quadratic polynomial 
\begin{equation}\label{decoration2} x^{2} - 2(2\Psi((\tilde{\Lambda})') + \Psi(\tilde{\Lambda}))x + \Psi((\tilde{\Lambda})')^{2} \in \bar{K}[x], \end{equation}
and for every $\Lambda \in \mathcal{L}_{n}$ with $n \geq 3$, $\Psi(\Lambda)$ is a root of the quadratic polynomial 
\begin{equation}\label{decoration>2} x^{2} - 2(\Psi((\tilde{\Lambda})') - 2\Psi(\tilde{\Lambda}))x + \Psi((\tilde{\Lambda})')^{2} \in \bar{K}[x]. \end{equation}

\end{dfn}

\begin{prop}\label{prop: decorations exist}

a) A decoration on $\mathcal{L}$ exists.

b) Let $\Psi$ be a decoration on $\mathcal{L}$.  For $n \geq 0$, let $K_{n}' = K(\{\Psi(\Lambda)\}_{\Lambda \in \mathcal{L}_{\leq n}})$ and let 
$$K_{\infty}' = \bigcup_{n \geq 0} K_{n}' = K(\{\Psi(\Lambda)\}_{\Lambda \in \mathcal{L}_{\geq 1}}).$$
Then the algebraic extensions $K_{n}' / K$ and $K_{\infty}' / K$ do not depend on the choice of $\Psi$.

\end{prop}

\begin{proof}

For each $N \geq 1$, define $F_{N}$ to be the set of all functions $\Psi : \mathcal{L}_{\geq 1} \cap \mathcal{L}_{\leq N} \to \bar{K}$ that satisfy I, II, and III for $n \leq N$.  Clearly, each $F_{N}$ is finite, and for each $N < N'$, there is a map from $F_{N'}$ to $F_{N}$ by restriction, so it will suffice to show that each $F_{N}$ is nonempty.  By definition, $F_{1}$ is nonempty, and one can explicitly show that $F_{2}$ is nonempty and that any function $\Psi \in F_{2}$ takes nonzero values in $\bar{K}$.  Now we prove inductively that $F_{N}$ is nonempty for $N \geq 3$ by showing that for each $N \geq 2$ and function $\Psi_{N} \in F_{N}$, there is a function $\Psi_{N + 1} \in F_{N + 1}$ taking nonzero values which restricts to $\Psi_{N}$.  This amounts to showing that for each $\Lambda \in \mathcal{L}_{n}$ with $n \geq 2$, the polynomial $x^{2} - 2(\Psi(\Lambda') - 2\Psi(\Lambda))x + \Psi(\Lambda')^{2}$ has two distinct, nonzero roots in $\bar{K}$.  It is clear from property II and a little computation that this is true for $\Lambda \in \mathcal{L}_{2}$.  Now assume inductively that this claim holds for all $\Lambda \in \mathcal{L}_{n - 1}$ for some $n \geq 3$.  Let $\Lambda \in \mathcal{L}_{n}$.  If $0$ were a root of $x^{2} - 2(\Psi(\Lambda') - 2\Psi(\Lambda))x + \Psi(\Lambda')^{2}$, then the constant coefficient $(\Psi(\Lambda'))^{2}$ would be $0$.  But $\Psi(\Lambda')$ is a root of the polynomial $x^{2} - 2(\Psi(\tilde{\Lambda}') - 2\Psi(\tilde{\Lambda}))x + \Psi(\tilde{\Lambda}')^{2}$, which by the inductive assumption, has nonzero roots.  Thus, the polynomial $x^{2} - 2(\Psi(\Lambda') - 2\Psi(\Lambda))x + \Psi(\Lambda')^{2}$ has nonzero roots.  Now suppose that its roots are equal.  Then its discriminant $4(\Psi(\Lambda') - 2\Psi(\Lambda))^{2} - 4\Psi(\Lambda')^{2} = 16\Psi(\Lambda)(\Psi(\Lambda) - \Psi(\Lambda'))$ is $0$, implying that either $\Psi(\Lambda) = 0$ or $\Psi(\Lambda) = \Psi(\Lambda')$.  But $\Psi(\Lambda)$ and $\Psi(\Lambda')$ are the two roots of the polynomial $x^{2} - 2(\Psi(\tilde{\Lambda}') - 2\Psi(\tilde{\Lambda}))x + \Psi(\tilde{\Lambda}')^{2}$, and by the inductive assumption, they are distinct and nonzero, so we have a contradiction, thus proving part (a).

Let $\Psi$ and $\Psi'$ be two functions on $\mathcal{L}_{\geq 1}$ taking values in $\bar{K}$, and choose any $n \geq 0$.  Then it is easy to show by induction on $n$ that if $\Psi$ and $\Psi'$ both satisfy I, II, and III, then there is a permutation $\Sigma$ on $\mathcal{L}_{n}$ such that $\Psi'(\Lambda) = \Psi(\Lambda^{\Sigma})$ for each $\Lambda \in \mathcal{L}_{n}$.   Therefore, $K(\{\Psi(\Lambda)\}_{\Lambda \in \mathcal{L}_{n}}) = K(\{\Psi'(\Lambda)\}_{\Lambda \in \mathcal{L}_{n}})$ for each $n$, which immediately implies part (b).

\end{proof}

We write $\rho_{2} : \Gal(\bar{K} / K) \to \GL(T_{2}(E)) = \Aut_{\zz_{2}}(T_{2}(E))$ for the continuous homomorphism induced by the natural Galois action on $T_{2}(E)$, and denote its image by $G$.  Similarly, for any integer $n \geq 0$, we write $\bar{\rho}_{2}^{(n)} : \Gal(K_{n} / K) \to \GL(E[2^{n}])$ for the homomorphism induced by the natural Galois action on $E[2^{n}]$, and denote its image by $\bar{G}^{(n)}$.  Let $G(n)$ denote the kernel of the natural surjection $G \twoheadrightarrow \bar{G}^{(n)}$; it is the image under $\rho_{2}$ of the normal subgroup $\Gal(\bar{K} / K_{n}) \lhd \Gal(\bar{K} / K)$.  Note that $G(0) = G$.

It follows from Corollary 1.2(b) and Remark 4.2(a) of \cite{yelton2014images} that $G$ contains the subgroup $\SL(T_{2}(E)) \subset GL(T_{2}(E))$ of automorphisms of determinant $1$.  (This also follows from applying Hilbert's Irreducibility Theorem to results such as Corollary 1 of Chapter 6, \S3 of \cite{lang1987elliptic}.)  We write $-1 \in \SL(T_{2}(E)) \subseteq G(1)$ for the scalar automorphism which acts on $T_{2}(E)$ as multiplication by $-1$.

We fix a compatible system $\{\zeta_{2^{n}}\}_{n \geq 0}$ of $2^{n}$-th roots of unity.  For any extension field $L$ of $K$, let $L(\mu_{2}) = \bigcup_{n = 1}^{\infty} L(\zeta_{2^{n}})$.

We are now ready to state the main theorem.

\begin{thm}\label{thm: main}

Let an elliptic curve $E / K$ be defined as above, with Weierstrass roots $\alpha_{1}$, $\alpha_{2}$, and $\alpha_{3}$, and define $K_{\infty}'$ as in Proposition \ref{prop: decorations exist}(b).

a) Choose $i, j \in \{1, 2, 3\}$ with $i \neq j$, and choose an element $\sqrt{\alpha_{i} - \alpha_{j}} \in \bar{K}$ whose square is $\alpha_{i} - \alpha_{j}$.  Then we have
$$K_{\infty} = K_{\infty}'(\sqrt{\alpha_{i} - \alpha_{j}})(\mu_{2}). $$

b) The element of $\Gal(\bar{K} / K_{1})$ whose image under $\rho_{2}$ is $-1 \in G(1)$ acts on $K_{\infty}$ by fixing $K_{\infty}'(\mu_{2})$ and taking $\sqrt{\alpha_{i} - \alpha_{j}}$ to $-\sqrt{\alpha_{i} - \alpha_{j}}$ for $1 \leq i, j \leq 3, i \neq j$.

\end{thm}

The next section is dedicated to a proof of Theorem \ref{thm: main}.  The idea behind the proof is to construct $K_{\infty}'$ as a compositum of fields of definition of certain elliptic curves that have a $2$-power isogeny to $E$.  It will then be shown that $K_{\infty}' \subset K_{\infty}$ is the subextension corresponding to the subgroup of scalar automorphisms in $G$, and generators of $K_{\infty}$ over $K_{\infty}'$ will be found.  In \S3, we will use Theorem \ref{thm: main} to obtain additional results (Theorems \ref{prop: x-coordinates} and \ref{prop: bounding K(2^n)}).

\section{Proof of the main theorem}

We assume the notation of \S1.  In particular, we retain the convention that for $\alpha_{i} \in \{\alpha_{1}, \alpha_{2}, \alpha_{3}\}$, we will write $\alpha_{i + 1}$ or $\alpha_{i + 2}$ as though $i \in \zz / 3\zz$.  For each $\Lambda \in \mathcal{L}$, by the discussion in \S1, there is a corresponding cyclic subgroup of $E[2^{\infty}]$, which we will denote by $N_{\Lambda}$.  Furthermore, we will often use ``$<$" and ``$>$" to indicate inclusion of subgroups of $E[2^{\infty}]$.

We will assign to each $\Lambda \in \mathcal{L}_{n}$ an elliptic curve $E_{\Lambda}$ and a $2^{n}$-isogeny $\phi_{\Lambda} : E \rightarrow E_{\Lambda}$ whose kernel is $N_{\Lambda}$, which we will later show (Proposition \ref{prop: field of definition of isogenies}(b)) is defined over $K(N_{\Lambda})$.  We will do this using a well-known isogeny of degree $2$ (\cite{silverman2009arithmetic}, Chapter III, Example 4.5) which is defined over the field of definition of its kernel.

Set $E_{\Lambda_{0}} := E$, and let $\phi_{\Lambda_{0}} : E \rightarrow E_{\Lambda_{0}}$ be the identity isogeny.

For any $z \in \bar{K}$, denote by $E_{z}$ the elliptic curve with Weierstrass equation given by 
\begin{equation}\label{translated elliptic curve} y^{2} = (x - \alpha_{1} - z)(x - \alpha_{2} - z)(x - \alpha_{3} - z). \end{equation}
Let $t_{z}$ the isomorphism $E \rightarrow E_{z}$ sending $(x, y) \mapsto (x + z, y)$.  Define, for any $\beta, \gamma \in K_{1}$, the elliptic curves $E_{\beta, \gamma}$ and $E'_{\beta, \gamma}$ given by the following Weierstrass equations:
\begin{equation} E_{\beta, \gamma} : y^{2} = x(x - \beta)(x - \gamma), \end{equation}
\begin{equation}\label{target of 2-isogeny}  E'_{\beta, \gamma}: y^{2} = x^{3} - 2(\beta + \gamma)x^{2} + (\beta - \gamma)^{2}x. \end{equation}
Let $\phi_{\beta, \gamma} : E_{\beta, \gamma} \rightarrow E'_{\beta, \gamma}$ be the isogeny of degree $2$ given by
\begin{equation}\label{2-isogeny} \phi_{\beta, \gamma} : (x, y) \mapsto (x - (\beta + \gamma) + \frac{\beta\gamma}{x}, y(1 - \frac{\beta\gamma}{x^{2}})).\end{equation}
Note that the kernel of $\phi_{\beta, \gamma}$ has as its only nontrivial element the $2$-torsion point $(0, 0) \in E_{\beta, \gamma}$.  Now define, for $i \in \zz /3\zz$, 
\begin{equation}\label{first 2-isogeny} \phi_{\Lambda(\alpha_{i})}: E \rightarrow E'_{\alpha_{i + 1} - \alpha_{i}, \alpha_{i + 2} - \alpha_{i}},\ \phi_{\Lambda(\alpha_{i})} := \phi_{\alpha_{i + 1} - \alpha_{i}, \alpha_{i + 2} - \alpha_{i}} \circ t_{-\alpha_{i}}. \end{equation}
We assign $E_{\Lambda(\alpha_{i})} := E'_{\alpha_{i + 1} - \alpha_{i}, \alpha_{i + 2} - \alpha_{i}}$.  Note that for each $i$, $\phi_{\Lambda(\alpha_{i})}$ is an isogeny whose kernel is the order-$2$ cyclic subgroup $N_{\Lambda(\alpha_{i})} = \langle (\alpha_{i}, 0) \rangle$ of $E$.

From now on, for each $i \in \zz/3\zz$, let $a_{\Lambda(\alpha_{i})} = \alpha_{i + 1} - \alpha_{i + 2}$.  Now we may write the Weierstrass equation for $E_{\Lambda_{i}}$ as 
$$ y^{2} = x^{3} - 2((\alpha_{i + 2} - \alpha_{i}) - (\alpha_{i} - \alpha_{i + 1}))x^{2} + (\alpha_{i + 1} - \alpha_{i + 2})^{2}x $$
$$ = x^{3} - 2(2(\alpha_{i + 2} - \alpha_{i}) + (\alpha_{i + 1} - \alpha_{i + 2}))x^{2} + (\alpha_{i + 1} - \alpha_{i + 2})^{2}x$$
\begin{equation}\label{decoration2 cubic} = x^{3} - 2(2a_{\Lambda(\alpha_{i})'} + a_{\Lambda(\alpha_{i})})x^{2} + a_{\Lambda(\alpha_{i})}^{2}x. \end{equation}

Since $0$ is a root of the cubic in the above equation, we know that $(0, 0) \in E_{\Lambda(\alpha_{i})}[2]$, and it is easy to verify that in fact, $(0, 0)$ is the image of both points in $E[2] \backslash N_{\Lambda(\alpha_{i})}$.  It follows that the inverse image of $\langle (0, 0) \rangle < E_{N_{\Lambda(\alpha_{i})}}[2]$ under $\phi_{N_{\Lambda(\alpha_{i})}}$ is $E[2]$.  Then the inverse images of the other two order-$2$ subgroups of $E_{N_{\Lambda(\alpha_{j})}}(\bar{K})$ under $\phi_{N_{\Lambda(\alpha_{j})}}$ are the two cyclic order-$4$ subgroups of $E(\bar{K})$ which contain $N_{\Lambda(\alpha_{j})}$.  It follows that these cyclic order-$4$ subgroups must be $N_{\Lambda}$ and $N_{\Lambda'}$, where $\Lambda$ and $\Lambda'$ are twin vertices in $\mathcal{L}_{2}$ whose parent vertex is $\Lambda(\alpha_{i})$.  Let $a_{\Lambda}$ (resp. $a_{\Lambda'}$) be the (nonzero) root of the cubic in the above equation such that $\phi_{N_{\Lambda(\alpha_{i})}}$ takes $N_{\Lambda}$ (resp. $N_{\Lambda'}$) to the subgroup $\langle (a_{\Lambda}, 0) \rangle$ (resp. $\langle (a_{\Lambda'}, 0) \rangle$) of $E_{N_{\Lambda(\alpha_{i})}}(\bar{K})$.  Now, using the notation of above, we have the elliptic curve $E'_{-a_{\Lambda}, a_{\Lambda'} - a_{\Lambda}}$ and the isogeny $\phi_{-a_{\Lambda}, a_{\Lambda'} - a_{\Lambda}} \circ t_{-a_{\Lambda}} : E_{N_{\Lambda(\alpha_{i})}} \to E'_{-a_{\Lambda}, a_{\Lambda'} - a_{\Lambda}}$.  Its kernel is $\langle a_{\Lambda}, 0 \rangle$.  Therefore, if we assign $E_{N_{\Lambda}} := E'_{-a_{\Lambda}, a_{\Lambda'} - a_{\Lambda}}$ and 
\begin{equation}\label{2-isogeny>2} \phi_{N_{\Lambda}} := \phi_{-a_{\Lambda}, a_{\Lambda'} - a_{\Lambda}} \circ t_{-a_{\Lambda}} \circ \phi_{N_{\Lambda(\alpha_{j})}} : E \to E_{N_{\Lambda}}, \end{equation}
then $\phi_{N_{\Lambda}}$ has kernel $N_{\Lambda}$.  Its Weierstrass equation can be written as 
$$y^{2} = x^{3} - 2((-a_{\Lambda}) + (a_{\Lambda'} - a_{\Lambda}))x^{2} + ((-a_{\Lambda}) - (a_{\Lambda'} - a_{\Lambda}))^{2}x$$
\begin{equation}\label{decoration>2 cubic} = x^{3} - 2(a_{\Lambda'} - 2a_{\Lambda})x^{2} + a_{\Lambda'}^{2}x. \end{equation}
Thus, we have defined the desired $E_{N_{\Lambda}}$ and $\phi_{N_{\Lambda}}$ for all $v \in \mathcal{L}_{2}$.

We define the desired $a_{\Lambda}$, $\phi_{\Lambda}$, and $E_{\Lambda}$ for any $\Lambda \in \mathcal{L}_{\geq 3}$ in a similar manner, using induction.  The idea is as follows.  Assume that for some $n \geq 2$ we have defined $\phi_{\Lambda}$, $E_{\Lambda}$, and $a_{\Lambda}$ for all $\Lambda \in \mathcal{L}_{n}$ with the above properties, and choose a vertex $\lambda \in \mathcal{L}_{n + 1}$ whose parent is $\Lambda$.  Then since the Weierstrass equation of $E_{\Lambda}$ is given by (\ref{decoration>2 cubic}), we may define $a_{\lambda}$ (resp. $a_{\lambda'}$) to be the (nonzero) root of the cubic in (\ref{decoration>2 cubic}) such that $\phi_{\Lambda}$ takes $N_{\lambda}$ (resp. $N_{\lambda'}$) to the subgroup $\langle (a_{\lambda}, 0) \rangle$ (resp. $\langle (a_{\lambda'}, 0) \rangle$) of $E_{\Lambda}(\bar{K})$.  We make the assignments $E_{\lambda} := E'_{-a_{\lambda}, a_{\lambda'} - a_{\lambda}}$ and 
$$\phi_{\lambda} := \phi_{-a_{\lambda}, a_{\lambda'} - a_{\lambda}} \circ t_{-a_{\lambda}} \circ \phi_{\Lambda} : E \to E_{\lambda},$$
and check that $a_{\lambda}$, $\phi_{\lambda}$, and $E_{\lambda}$ have the desired properties (in particular, the kernel of $\phi_{\lambda}$ is $N_{\lambda}$).  Since the parent of every vertex $\mathcal{L}_{n + 1}$ is a vertex in $\mathcal{L}_{n}$, it follows that through the method described above, we have defined the desired $E_{\Lambda}$ and $\phi_{\Lambda}$ for all $\Lambda \in \mathcal{L}_{n + 1}$.  In this way, $a_{\Lambda}$, $E_{\Lambda}$, and $\phi_{\Lambda} \in \bar{K}$ are defined for all $\Lambda \in \mathcal{L}_{\geq 1}$.  Furthermore, for all $\Lambda \in \mathcal{L}_{\geq 1}$, we define $K_{\Lambda}$ to be the extension of $K$ obtained by adjoining the coefficients of the Weierstrass equation of $E_{\Lambda}$ given above.

\begin{lemma}\label{prop: Phi is a decoration}

Using the above notation, define $\Psi : \mathcal{L}_{\geq 1} \to \bar{K}$ by setting $\Psi(\Lambda) = a_{\Lambda}$ for $\Lambda \in {\mathcal{L}}_{\geq 1}$.  Then $\Psi$ is a decoration on $\mathcal{L}$.

\end{lemma}

\begin{proof}

By construction, $\Psi$ satisfies properties II and III in Definition \ref{dfn decoration} (see the cubics in (\ref{decoration2 cubic}) and (\ref{decoration>2 cubic})).  Finally, as in the proof of Proposition \ref{prop: decorations exist}, the roots of the above quadratics must be distinct, fulfilling property I.

\end{proof}

\begin{dfn}\label{dfn K'}
For any integer $n \geq 0$, define the extension $K_{n}'$ of $K$ to be the compositum of the fields $K_{N_{v}}$ for all $v \in |\mathcal{L}|_{\leq n} \backslash \{v_{0}\}$.  Define the extension $K_{\infty}'$ of $\bar{K}$ to be the infinite compositum 
$$K_{\infty}' := \bigcup_{n \geq 0} K_{n}'.$$
\end{dfn}

In this way, we obtain a tower of field extensions
\begin{equation}\label{tower of fields'} K = K_{0}' \subset K_{1}' \subset K_{2}' \subset ... \subset K_{n}' \subset ..., \end{equation}
with $K_{\infty}' = \bigcup_{n \geq 0} K_{n}'$.  The following lemma shows that this notation is consistent with the notation set in the statement of Proposition \ref{prop: decorations exist}(b).

\begin{lemma}\label{prop: field of definition}

 For any $n \geq 1$, $K_{n}' = K(\{\Psi(v)\}_{\Lambda \in \mathcal{L}_{\leq n} \backslash \{\Lambda_{0}\}})$ for any decoration $\Psi$ on $\mathcal{L}$, and $K_{\infty}' = K(\{\Psi(v)\}_{\Lambda \in \mathcal{L}_{\geq 1}})$ for any decoration $\Psi$ on $\mathcal{L}$.

\end{lemma}

\begin{proof}

For each $\Lambda \in |\mathcal{L}|_{n}$ with $n \geq 1$, let $\{\Lambda_{0}, \Lambda_{1}, ... , \Lambda_{n} = \Lambda\}$ be the sequence of vertices in the path of length $n$ from $v_{0}$ to $v$.  Let $\tilde{K}_{\Lambda}$ denote the compositum of the fields $K_{\Lambda}$ for all $\Lambda \in \{\Lambda_{0}, \Lambda_{1}, ... , \Lambda_{n}\}$.  We claim that  
$$\tilde{K}_{\Lambda} = K(\alpha_{1}, \alpha_{2}, \alpha_{3}, \{a_{\Lambda_{m}}\}_{1 \leq m \leq n}).$$
By Definition \ref{dfn K'}, Propositions \ref{prop: decorations exist}(b), and Lemma \ref{prop: Phi is a decoration}, this suffices to prove the statement of the lemma.

The claim is trivial for $n = 1$.  Now assume inductively that the statement holds for some $n \geq 1$ and all $\Lambda \in \mathcal{L}_{n}$.  Choose any $\Lambda \in \mathcal{L}_{n + 1}$.  We may apply the inductive assumption to $\tilde{\Lambda} \in \mathcal{L}_{n}$.  We know that $E_{\Lambda}$ is given by a Weierstrass equation of the form (\ref{decoration2 cubic}) or (\ref{decoration>2 cubic}) and is therefore defined over $K(a_{v}, a_{v'})$.  But $a_{v}a_{v'}$ is a coefficient of $E_{\tilde{\Lambda}}$, and so the only element that we need to adjoin to $\tilde{K}_{\tilde{\Lambda}} = K(\alpha_{1}, \alpha_{2}, \alpha_{3}, \{a_{\Lambda_{m}}\}_{1 \leq m \leq n - 1})$ to obtain $\tilde{K}_{\Lambda}$ is $a_{\Lambda}$.  Moreover, $a_{\Lambda}$ does lie in this extension, since $-(a_{\Lambda} + a_{\Lambda'})$ is a coefficient in the equation for $E_{\Lambda}$ and $2(2a_{\Lambda'} + a_{\Lambda})$ (resp. $2(a_{\Lambda'} - 2a_{\Lambda})$) is a coefficient of $E_{\tilde{\Lambda}}$ if $n = 1$ (resp. $n \geq 2$).  Thus, we have proved the claim for $n + 1$.

\end{proof}

\begin{prop}\label{prop: field of definition of isogenies} With the above notation, 

a) the isogeny $\phi_{N_{v}}$ is defined over $K(N_{v})$, and $K_{N_{v}} \subseteq K(N_{v})$,

b) for all $n \geq 0$, $K_{n}' \subseteq K_{n}$, and equality holds for $n = 0, 1$.

\end{prop}

\begin{proof}

First of all, for $i \in \zz / 3\zz$, $E_{\Lambda(\alpha_{i})}$ and $\phi_{\Lambda(\alpha_{i})}$ are defined over $K(\alpha_{i + 1} - \alpha_{i}, \alpha_{i + 2} - \alpha_{i}) = K(\alpha_{i})$.  This implies the equality in the $n = 1$ case of the statament in part (b) (the equality in the $n = 0$ case is trivial).  It also proves part (a) for $\Lambda \in \mathcal{L}_{1}$, since $K(N_{\Lambda(\alpha_{i})}) = K(\alpha_{i})$ for each $i \in \zz / 3\zz$.

Now assume inductively that for some $n \geq 1$ and all $\Lambda \in \mathcal{L}_{n}$, $\phi_{\Lambda}$ is defined over $K(N_{\Lambda})$ and $K_{N_{\Lambda}} \subseteq K(N_{\Lambda})$.  Choose any $\Lambda \in \mathcal{L}_{n + 1}$.  We may apply the inductive assumption to $\tilde{\Lambda}$, since $\tilde{\Lambda} \in \mathcal{L}_{n}$.  Let $P$ be a generator of the cyclic order-$2^{n + 1}$ subgroup $N_{\Lambda}$.  Then $P$ has coordinates in $K(N_{\Lambda})$ and $\phi_{\tilde{\Lambda}}$ is defined over $K(N_{\tilde{\Lambda}}) \subseteq K(N_{\Lambda})$, and it follows that $\phi_{\tilde{\Lambda}}(P) = (a_{\Lambda}, 0)$ has coordinates in $K(N_{\Lambda})$.  Thus, $a_{\Lambda} \in K(N_{\Lambda})$.  But $a_{\Lambda}a_{\Lambda'}$ is a coefficient of $E_{\tilde{\Lambda}}$ and $K_{\tilde{\Lambda}} \subseteq K(N_{\tilde{\Lambda}})$, so $a_{\Lambda'} \in K(N_{\Lambda})$ also.  By construction, $\phi_{-a_{\Lambda}, a_{\Lambda'} - a_{\Lambda}} \circ t_{-a_{\Lambda}}$ is defined over $K(a_{\Lambda}, a_{\Lambda'})$, so $\phi_{\Lambda} = \phi_{-a_{\Lambda}, a_{\Lambda'} - a_{\Lambda}} \circ t_{-a_{\Lambda}} \circ \phi_{\tilde{\Lambda}}$ is defined over $K(N_{\tilde{\Lambda}})(a_{\Lambda}, a_{\Lambda'}) \subseteq K(N_{\Lambda})$.  Moreover, the Weierstrass equation (\ref{decoration>2 cubic}) of $E_{\Lambda} = E'_{-a_{\Lambda}, a_{\Lambda'} - a_{\Lambda}}$ has coefficients in $K(a_{\Lambda}, a_{\Lambda'}) \subseteq K(N_{\Lambda})$, and so $K_{\Lambda} \subseteq K(N_{\Lambda})$, thus proving part (a).

Now part (a) and the fact that $K_{n}'$ is the compositum of the fields $K_{\Lambda}$ for all $\Lambda \in \mathcal{L}_{\leq n} \backslash \{\Lambda_{0}\}$ imply that $K_{n}'$ is contained in the compositum of the extensions $K(N_{\Lambda})$ for all $\Lambda \in \mathcal{L}_{\leq n} \backslash \{\Lambda_{0}\}$.  Since $\{N_{\Lambda}\}_{\Lambda \in \mathcal{L}_{\leq n}}$ is the set of all cyclic subgroups of $E[2^{n}]$ and therefore generates $E[2^{n}]$, this compositum is $K_{n}$.  Thus, $K_{n}' \subseteq K_{n}$, which is the statement of (b).

\end{proof}

Next we want to determine how the absolute Galois group of $K$ acts on the $a_{\Lambda}$'s defined above.  In order to describe this Galois action, we will adopt the following notation.  The automorphism group $\GL(T_{2}(E))$ acts on the set of rank-$2$ $\zz_{2}$-lattices in $V_{2}(E)$ by left multiplication, and this action stabilizes $\mathcal{L}$, since $\GL(T_{2}(E))$ fixes $T_{2}(E)$.  (In fact, this action of $\GL(T_{2}(E))$ on $\mathcal{L}$ is the action of $\Aut_{\zz_{2}}(T_{2}(E)) \subset \mathrm{Aut}_{\qq_{2}}(V_{2}(E))$ on the Bruhat-Tits tree as described in \cite{serre2003trees}, \S1.2.)  In particular, each $\mathcal{L}_{n}$ is invariant under the action.  Recall the definitions of $G$ and $\bar{G}^{(n)}$ from \S1.  We will denote the action of $G \subset \GL(T_{2}(E))$ on $\Lambda$ by $(s, \Lambda) \mapsto s \cdot \Lambda$ for an automorphism $s$ and vertex $\Lambda$.  Note that this action of $G$, when restricted to $\mathcal{L}_{\leq n}$, factors through $G \twoheadrightarrow \bar{G}^{(n)}$.  We similarly denote the resulting action of $\bar{G}^{(n)}$ on the subtree $\Lambda_{n}$ by $(\bar{s}, \Lambda) \mapsto \bar{s} \cdot \Lambda$ for an automorphism $\bar{s}$ and a vertex $\Lambda$.

Recall that $\rho_{2} : \mathrm{Gal}(\bar{K} / K) \to \GL(T_{2}(E))$ (resp. $\bar{\rho}_{2}^{(n)} : \mathrm{Gal}(K_{n} / K) \to \GL(E[2^{n}])$) is the homomorphism induced by the Galois action on $T_{2}(E)$ (resp. the Galois action on $E[2^{n}]$, for each $n \geq 0$).  For any Galois element $\sigma \in \mathrm{Gal}(\bar{K} / K)$ and vertex $\Lambda$ of $\mathcal{L}$, let $\Lambda^{\sigma} := \rho_{2}(\sigma) \cdot \Lambda$.  If $\Lambda \in \mathcal{L}_{\leq n}$ for some $n \geq 1$, then let $\Lambda^{\sigma | K_{n}} := \bar{\rho}_{2}^{(n)}(\sigma) \cdot \Lambda$.

\begin{lemma}\label{prop: Galois action on decoration}

For any $\sigma \in \mathrm{Gal}(\bar{K} / K_{1})$ and vertex $\Lambda$, we have $a_{v}^{\sigma} = a_{v^{\sigma}}$.  If $\Lambda \in \mathcal{L}_{\leq n}$, then $a_{\Lambda}^{\sigma |_{K_{n}}} = a_{\Lambda^{\sigma | K_{n}}}$.

\end{lemma}

\begin{proof}

Choose any $\sigma \in \mathrm{Gal}(\bar{K} / K_{1})$.  We will prove that $a_{\Lambda}^{\sigma} = a_{\Lambda^{\sigma}}$ for all $\Lambda \in \mathcal{L}_{n}$ for each $n \geq 1$.  The claim is trivially true for $n = 1$.  Moreover, in the $n = 1$ case, $E_{\Lambda}$ and $\phi_{\Lambda}$ are clearly defined over $K_{1}$ and are therefore fixed by $\sigma$.  In particular, for any $\Lambda \in \mathcal{L}_{1}$, $a_{\Lambda}^{\sigma} = a_{\Lambda^{\sigma}}$, $E_{\Lambda}^{\sigma} = E_{\Lambda^{\sigma}}$, and $\phi_{\Lambda}^{\sigma} = \phi_{\Lambda^{\sigma}}$.  Now choose $\Lambda \in \mathcal{L}_{2}$.  Then $a_{\Lambda}$ is a nonzero root of (\ref{decoration2 cubic}), which is the Weierstrass cubic for $E_{\tilde{\Lambda}}$.  Let $P$ be a generator of the cyclic order-$4$ subgroup $N_{\Lambda}$; then $\phi_{\tilde{\Lambda}}(P) = (a_{\Lambda}, 0) \in E_{\tilde{\Lambda}}[2]$, by the above construction of $a_{\Lambda}$.  So we have 
\begin{equation}\label{proving Galois action on decoration2} \phi_{N_{\tilde{v}^{\sigma}}}(P^{\sigma}) = \phi_{N_{\tilde{v}}}^{\sigma}(P^{\sigma}) = (\phi_{N_{\tilde{v}}}(P))^{\sigma} = (\alpha_{v}^{\sigma}, 0). \end{equation}
But $P^{\sigma}$ generates $N_{v}^{\sigma} = N_{v^{\sigma}}$.  By the above construction, we have $\phi_{N_{\tilde{v}^{\sigma}}}(P^{\sigma}) = (a_{v^{\sigma}}, 0)$.  Then (\ref{proving Galois action on decoration2}) implies that $a_{v}^{\sigma} = a_{v^{\sigma}}$.

Now assume inductively that for some $n \geq 2$ and $\sigma \in \mathcal{L}_{n}$, $a_{\Lambda}^{\sigma} = a_{\Lambda^{\sigma}}$, $E_{\Lambda}^{\sigma} = E_{v^{\sigma}}$, and $\phi_{\Lambda}^{\sigma} = \phi_{\Lambda^{\sigma}}$, and choose any $\Lambda \in \mathcal{L}_{n + 1}$.  Then $a_{\Lambda}$ is a nonzero root of (\ref{decoration>2 cubic}), and by a similar argument, again $a_{\Lambda}^{\sigma} = a_{\Lambda^{\sigma}}$.  This inductive argument proves the first statement, and the second statement follows immediately.
\end{proof}

\begin{prop}\label{prop: key lemma}

For all $n \geq 1$, the image of $\mathrm{Gal}(K_{n}/K_{n}')$ (resp. $\Gal(K_{\infty} / K_{\infty}')$) under $\bar{\rho}_{2}^{(n)}$ (resp. $\rho_{2}$) coincides with the subgroup of scalar automorphisms in $G$ (resp. in $\bar{G}^{(n)}$).

\end{prop}

\begin{rmk}

Since each $\sigma \in \Gal(\bar{K} / K)$ takes $\zeta_{8}$ to $\zeta_{8}^{\det(\bar{\rho}_{2}^{(3)}(\sigma))}$, all scalar automorphisms in $G$ must fix $\zeta_{8}$.  Thus, Proposition \ref{prop: key lemma} implies that $K_{\infty}$ contains the $8$th roots of unity.

\end{rmk}

\begin{proof}

Fix $n \geq 1$.  Since $K_{n}' \supseteq K_{1}$ for each $n \geq 1$, we only need to consider the Galois subgroup $\mathrm{Gal}(K_{n}' / K_{1}) \subseteq \Gal(K_{n} / K_{1})$.  Proposition \ref{prop: field of definition}, with the help of Lemma \ref{prop: Phi is a decoration}, implies that $K_{n}'$ is generated over $K_{1}$ by the elements $a_{\Lambda}$ for all $\Lambda \in \mathcal{L}_{\leq n} \backslash \{\Lambda_{0}\}$.  Therefore, the elements of $\mathrm{Gal}(K_{\infty} / K_{1})$ which fix $K_{n}'$ are exactly those which fix all of the elements $a_{\Lambda}$ for $\Lambda \in \mathcal{L}_{\leq n} \backslash \{\Lambda_{0}\}$.  By Lemma \ref{prop: Galois action on decoration}, these are the Galois elements which fix every vertex in $\mathcal{L}_{\leq n}$.  It is easy to see that an element of $\GL(E[2^{n}])$ fixes every vertex in $\mathcal{L}_{n}$ if and only if it is a scalar automorphism.  Thus, the image of $\Gal(K_{n} / K_{n}')$ coincides with the subgroup of scalars, as desired.  A similar argument proves the analogous statement for $\Gal(K_{\infty} / K_{\infty}')$.

\end{proof}

From now on, for ease of notation, we set $a_{i} := a_{\Lambda(\alpha_{i})} = \alpha_{i + 1} - \alpha_{i + 2}$ for $i \in \zz / 3\zz$.  For each $a_{i}$, choose an element $\sqrt{a_{i}} \in \bar{K}$ whose square is $a_{i}$.  Also, for $r \in \zz_{2}$, we will write $r \in \mathrm{SL}(T_{2}(E))$ (resp. $r \in \mathrm{SL}(E[2^{n}])$ for some $n$) for the scalar matrix corresponding to $r$ (resp. $r$ modulo $2^{n}$).

Note that, due to the Galois equivariance of the Weil pairing, the image of $\Gal(\bar{K} / K(\mu_{2}))$ under $\rho_{2}$ coincides with $\SL(T_{2}(E))$, and $K_{\infty} \supset K(\mu_{2})$.  The following proposition, together with Proposition \ref{prop: field of definition}, gives the statement of Theorem \ref{thm: main}.

\begin{prop}\label{prop: main quadratic extension}

The extension $K_{\infty}'(\mu_{2}) \supset K(\mu_{2})$ corresponds to the subgroup $\{\pm 1\} \lhd \mathrm{SL}(T_{2}(E)) \cong \Gal(K_{\infty} / K(\mu_{2}))$.  In fact, 
\begin{equation}\label{main quadratic extension} K_{\infty} = K_{\infty}'(\sqrt{a_{i}})(\mu_{2}) \end{equation}
for any $i \in \{1, 2, 3\}$, and the Galois element corresponding to $-1 \in \SL(T_{2}(E))$ acts by taking $\sqrt{a_{i}}$ to $-\sqrt{a_{i}}$.

\end{prop}

\begin{proof}

If we replace $K$ with $K(\mu_{2})$, it will suffice to assume that $K$ contains all $2$-power roots of unity and to prove that $K_{\infty} = K_{\infty}'(\sqrt{a_{i}})$ for any $i \in \{1, 2, 3\}$ and that the Galois element corresponding to $-1 \in \SL(T_{2}(E))$ acts as claimed.

Since $T_{2}(E)$ has rank $2$, the determinant of any scalar $r \in \Aut_{\zz_{2}}(T_{2}(E))$ is $r^{2} \in \zz_{2}$.  Therefore, the only scalar automorphisms in $\SL(T_{2}(E))$ are $\pm 1$.  Proposition \ref{prop: key lemma} then implies that the image under $\rho_{2}$ of $\Gal(K_{\infty} / K_{\infty}')$ coincides with $\{\pm 1\} \lhd \mathrm{SL}(T_{2}(E))$.

It immediately follows that $K_{\infty}$ is generated over $K_{\infty}'$ by any element of $K_{\infty}$ which is not fixed by the Galois automorphism $\sigma$ such that $\rho_{2}(\sigma) = -1 \in \mathrm{SL}(T_{2}(E))$.  Clearly, $\bar{\rho}_{2}^{(2)}(\sigma |_{K_{2}}) = -1 \in \Gamma(2) / \Gamma(4)$.  But setting $n = 2$ in the statement of Proposition \ref{prop: key lemma} implies that
\begin{equation} \label{Gal(K_{2}/K_{2}')} \mathrm{Gal}(K_{2} / K_{2}') \cong \{\pm 1\} < \Gamma(2) / \Gamma(4), \end{equation}
so any element in $K_{2} \setminus K_{2}'$ will not be fixed by $-1$.  One checks using the formulas for $a_{\Lambda}$ with $v \in |\mathcal{L}|_{2}$ and using Proposition \ref{prop: field of definition} for $n = 2$ that 
\begin{equation}\label{description of K(4)'} K_{2}' = K_{1}(\sqrt{a_{1}a_{2}}, \sqrt{a_{2}a_{3}}, \sqrt{a_{3}a_{1}}). \end{equation}
However, it can be verified through direct calculation that 
\begin{equation}\label{description of K(4)} K_{2} = K_{1}(\sqrt{\alpha_{1}}, \sqrt{\alpha_{2}}, \sqrt{\alpha_{3}}). \end{equation}
Therefore, for each $i$, $\sqrt{a_{i}} \in K_{2} \setminus K_{2}'$, so $\sqrt{a_{i}}$ can be used to generate $K_{\infty}$ over $K_{\infty}'$.  Moreover, since $a_{i} \in K_{1} \subset K_{\infty}'$, it follows that $-1$ sends $\sqrt{a_{i}}$ to $-\sqrt{a_{i}}$.

\end{proof}

\section{Some applications}

\begin{thm}\label{prop: x-coordinates}

Write $K(x(E[2^{n}]))$ (resp. $K(x(E[2^{\infty}]))$) for the extension of $K$ obtained by adjoining the $x$-coordinates of all elements of $E[2^{n}]$ (resp. $E[2^{\infty}]$).  Then 

a) $K_{n} = K(x(E[2^{n}]))(\sqrt{a_{i}})$ for all $n \geq 2$ and $i = 1, 2, 3$; 

b) $K_{n}'(\zeta_{2^{n}}) \subseteq K(x(E[2^{n}]))$ for all $n \geq 1$; and 

c) $K_{\infty}'(\mu_{2}) = K(x(E[2^{\infty}]))$.

\end{thm}

\begin{proof}

The Galois equivariance of the Weil pairing implies that $K_{n} \supset K(\zeta_{2^{n}})$ for each $n \geq 0$ and that $K_{\infty} \supset K(\mu_{2})$.  Choose a basis of $T_{2}(E)$, so that we may identify $\Gal(K_{\infty} / K)$ with $\GL_{2}(\zz_{2})$.  For any $n \geq 1$, the subgroup of $\mathrm{Gal}(K_{n} / K)$ which fixes the $x$-coordinates of the points in $E[2^{n}]$ coincides with the subgroup whose image under $\bar{\rho}_{2}^{(n)}$ is identified with the matrices in $\mathrm{GL}_{2}(\zz_{2})$ which send each point $P \in E[2^{n}]$ either to $P$ or to $-P$.  The only such matrices are the scalar matrices $\pm 1$.  Thus, $K(x(E[2^{n}]))$ is the subextension of $K_{n}$ fixed by $\{\pm 1\} \lhd \mathrm{GL}_{2}(\zz / 2^{n}\zz)$, and similarly, $K(x(E[2^{\infty}]))$ is the subextension of $K_{\infty}$ fixed by $\{\pm 1\} \lhd \mathrm{GL}(T_{2}(E))$.   As in the proof of Proposition \ref{prop: main quadratic extension}, $\sqrt{a_{i}}$ lies in $K_{2} \subseteq K_{n}$ and is not fixed by $-1$ for each $i$, hence the statement of part (a).  Parts (b) and (c) then immediately follow from Proposition \ref{prop: key lemma}.

\end{proof}

The description of $K_{\infty}'$ given in the statement of Theorem \ref{thm: main} provides us with recursive formulas for the generators of $K_{n}'$ for each $n \geq 0$.  We will not similarly obtain formulas for the generators of each extension $K_{n}$, but the above results do give us a way of ``bounding" each $K_{n}$, as follows:

\begin{thm}\label{prop: bounding K(2^n)}

For each $n \geq 2$ and $i = 1, 2, 3$,
$$K_{n}'(\sqrt{a_{i}}, \zeta_{2^{n}}) \subseteq K_{n} \subsetneq K_{n + 1}'(\sqrt{a_{i}}, \zeta_{2^{n + 1}}),$$
where the first inclusion is an equality if and only if $n = 2$.  Furthermore, $[K_{n} : K_{n}'(\sqrt{a_{1}}, \zeta_{2^{n}})] = 2$ for $n \geq 3$, and $[K_{n + 1}'(\sqrt{a_{1}}, \zeta_{2^{n + 1}}) : K_{n}] = 4$ for $n \geq 2$.

\end{thm}

\begin{proof}

Fix $n \geq 2$ and $i \in \{1, 2, 3\}$.  Proposition \ref{prop: field of definition of isogenies}(b) and the inclusion $K_{n} \supset K(\zeta_{2^{n}})$ imply that $K_{n}' \subset K_{n}$, and, as in the proof of Proposition \ref{prop: main quadratic extension}, $\sqrt{a_{i}} \in K_{2} \subseteq K_{n}$, thus implying the first inclusion.  For $n = 2$, it has already been shown in the proof of Proposition \ref{prop: main quadratic extension} that the inclusion is an equality.  Since $\sqrt{a_{i}} \notin K_{\infty}'(\mu_{2})$, it follows that $\sqrt{a_{i}} \notin K(2^{i})'(\zeta_{2^{i}})$ for any positive integer $i$.  Note that by Proposition \ref{prop: key lemma}, for $n \geq 3$, $\mathrm{Gal}(K_{n} / K_{n}'(\zeta_{2^{n}}))$ is identified with the subgroup $\{\pm 1, \pm (2^{n - 1} + 1) \} \lhd \mathrm{SL}_{2}(\zz / 2^{n}\zz)$, which is of order $4$.  It follows that the degree of the first inclusion is $2$ in this case.  Equivalently, for all $n \geq 2$, $K_{n + 1} \supset K_{n + 1}'(\sqrt{a_{1}}, \zeta_{2^{n + 1}})$ is an extension of degree $2$.  We have (via $\bar{\rho}_{2}^{(n + 1)}$) the following identifications: 
\begin{equation}\label{bounding K(2^n) identification 1} \mathrm{Gal}(K_{n + 1} / K) \cong G(2) / G(2^{n + 1}), \end{equation}
\begin{equation}\label{bounding K(2^n) identification 2} \mathrm{Gal}(K_{n + 1} / K_{n + 1}'(\zeta_{2^{n + 1}}) \cong \langle -1, 2^{n} + 1 \rangle \lhd G(2) / G(2^{n + 1}). \end{equation}
These imply that $\mathrm{Gal}(K_{n + 1} / K_{n + 1}'(\sqrt{a_{i}}, \zeta_{2^{n + 1}}))$ is a subgroup of $\langle -1, 2^{n} + 1 \rangle$ of order $2$.  Since $2^{n} + 1$ fixes all of $K(4)$, which includes the element $\sqrt{a_{i}}$, we have $\mathrm{Gal}(K_{n + 1} / K_{n + 1}'(\sqrt{\alpha_{i}}, \zeta_{2^{n + 1}})) \cong \langle 2^{n} + 1 \rangle \lhd \Gamma(2) / \Gamma(2^{n + 1})$.  But this subgroup also leaves $K_{n}$ fixed, whence the second inclusion $K_{n} \subset K_{n + 1}'(\sqrt{a_{i}}, \zeta_{2^{n + 1}})$.  The fact that $[K_{n + 1}'(\sqrt{a_{i}}) : K_{n}] = 4$ follows quickly from the fact that $\mathrm{Gal}(K_{n + 1} / K_{n}) \cong (G(n) \cap \SL(T_{2}(E)) / G(2^{n + 1})$ is isomorphic to $(\zz / 2\zz)^{3}$ (see the proof of Corollary 2.2 of \cite{sato2010abelianization}) and therefore has order $8$.

\end{proof}

\section*{Acknowledgements}

The author would like to thank Yuri Zarhin and Mihran Papikian for their patience, and for insightful discussions which were extremely helpful in producing this work.

\bibliographystyle{plain}
\bibliography{bibfile}

\end{document}